\documentclass[12pt,a4paper]{article}
\usepackage{latexsym,amssymb,amsfonts,amsmath,amsthm,nccmath,float,enumitem,setspace,bm,authblk,tabularx,longtable}
\usepackage[usenames,dvipsnames]{xcolor}
\usepackage[hidelinks]{hyperref}
\usepackage[margin=2cm]{geometry}

\hypersetup{
	colorlinks,
	linkcolor={red!80!black},
	citecolor={blue!80!black},
	urlcolor={blue!80!black}
}

\newtheorem{Theorem}{Theorem}[section]

\newtheorem{Lemma}{Lemma}[section]

\newtheorem{Definition}{Definition}

\def \leq {\leqslant}
\def \geq {\geqslant}

\let\oldproofname=\proofname
\renewcommand{\proofname}{\rm\bf{\oldproofname}}

\begin{document}

\title{The maximum size of the partial ground set of skew Bollob\'{a}s systems}
\author[a]{Yu Fang}
\author[b,c]{Tao Feng}
\author[a]{Xiaomiao Wang}
\affil[a]{School of Mathematics and Statistics, Ningbo University, Ningbo 315211, P. R. China}
\affil[b]{School of Mathematics and Statistics, Beijing Jiaotong University, Beijing 100044, P. R. China}
\affil[c]{Hebei Provincial Key Laboratory of Mathematical Theory and Analysis for Network and Data Science, Beijing Jiaotong University, Beijing, 100044, P. R. China}
\affil[ ]{fangymath@163.com; tfeng@bjtu.edu.cn; wangxiaomiao@nbu.edu.cn}
\renewcommand*{\Affilfont}{\small\it}
\renewcommand\Authands{ and }
\date{}

\maketitle

\footnotetext{Supported by NSFC under Grant 12271023 (T. Feng), and Ningbo Natural Science Foundation under Grant 2024J018 (X. Wang).}

\maketitle

\begin{abstract}
A skew Bollob\'{a}s system $\mathcal{P}=\{(A_i,B_i):1\leq i\leq m\}$ is a collection of pairs of disjoint subsets of $[n]$ such that $A_i\cap B_j\ne\emptyset$ for any $1\leq i<j\leq m$. Denote by $S_1(a, b)$ or $S_2(a, b)$ the maximum size of $\bigcup_{i=1}^m A_i$ or $\bigcup_{i=1}^m B_i$, respectively, over all possible skew Bollob\'{a}s systems $\mathcal{P}=\{(A_i,B_i):1\leq i\leq m\}$ satisfying $|A_i| \leq a$ and $|B_i| \leq b$ for all $i \in [m]$. It is shown that for any non-negative integers $a$ and $b$, $S_1(a,b)=\binom{a+b+1}{a}-1$ and $S_2(a,b)=\binom{a+b+1}{a+1}-1$.
\end{abstract}


\section{Introduction}

Throughout this paper, for a positive integer $n$, write $[n]=\{1,\dots,n\}$.	
\begin{Definition}
	Let $\mathcal{P}=\{(A_i,B_i):1\leq i\leq m\}$, where $A_i,B_i\subseteq [n]$ and $A_i\cap B_i=\emptyset$ for all $i\in[m]$. Then $\mathcal{P}$ is called a \emph{skew Bollob\'{a}s system} if $A_i\cap B_j\ne\emptyset$ for any $1\leq i<j\leq m$.
\end{Definition}

A skew Bollob\'{a}s system is also called a \emph{skew intersecting set pair system} (cf. \cite{gp}). Using linear algebraic methods, Frankl \cite{frankl} and Kalai \cite{kalai} gave the upper bound for the size of skew Bollob\'{a}s systems.

\begin{Theorem}{\rm \cite{frankl,kalai}}\label{skew-number}
Let $\mathcal{P}=\{(A_i,B_i):1\leq i\leq m\}$ be a skew Bollob\'{a}s system with $|A_i|\leq a$ and $|B_i|\leq b$ for every $i\in[m]$. Then
$$m\leq\binom{a+b}{a}.$$
\end{Theorem}

In this paper, we examine the maximum sizes of the partial ground sets $\bigcup_{i=1}^{m} A_i$ and $\bigcup_{i=1}^{m} B_i$ for skew Bollob\'{a}s systems $\mathcal{P}=\{(A_i,B_i):1\leq i\leq m\}$. Let

\begin{align*}
S_1(a,b)=&\max\{|\bigcup_{i=1}^{m} A_i|:	\{(A_i,B_i): 1\leq i\leq m\} \ \text{is a skew Bollob\'{a}s}\\
& \ \ \ \ \ \ \ \ \ \ \ \text{ system with } |A_i|\leq a,\ |B_i|\leq b \text{ for } i\in [m]\},
\end{align*}
and
\begin{align*}
S_2(a,b)=&\max\{|\bigcup_{i=1}^{m} B_i|:	\{(A_i,B_i): 1\leq i\leq m\} \ \text{is a skew Bollob\'{a}s}\\
& \ \ \ \ \ \ \ \ \ \ \ \text{ system with } |A_i|\leq a,\ |B_i|\leq b \text{ for } i\in [m]\}.
\end{align*}
The concept of $S_1(a,b)$ was proposed by Gerbner and Patk\'{o}s \cite[Page 68]{gp}, who referred to it as $n^{\prime}_1(a,b)$. To our knowledge, no literature has established the exact value of $S_1(a,b)$.

Given two real valued functions $f$ and $g$ of several variables, write $f\ll g$ if there exists a positive constant $C$ such that $f\leq Cg$ for all possible values of the variables. Write $f\asymp g$ if $f\ll g$ and $g\ll f$. The following asymptotics of $S_2(a,b)$ was determined by Calbet \cite[Theorem 27]{c}, where $S_2(a,b)$ was referred to as $u_s(a,b)$.

\begin{Theorem}{\rm \cite{c}}\label{known-T}
For all integers $a\geq 0$ and $b\geq 1$, $S_2(a,0)=0$ and $S_2(a,b)\asymp \binom{a+b+1}{a+1}$.
\end{Theorem}

In this work, we establish the exact values of $S_1(a,b)$ and $S_2(a,b)$.
	
\begin{Theorem}\label{main-rssult}
Let $a$ and $b$ be non-negative integers. Then
$$S_1(a,b)=\binom{a+b+1}{a}-1.$$
\end{Theorem}

\begin{Theorem}\label{2nd-result}
Let $a$ and $b$ be non-negative integers. Then
$$S_2(a,b)=\binom{a+b+1}{a+1}-1.$$
\end{Theorem}

\section{Proofs of Theorems \ref{main-rssult} and \ref{2nd-result}}

First of all, we establish the upper bound for $S_1(a,b)$.
		
\begin{Lemma}\label{upper bound}
Let $a$ and $b$ be positive integers. Then
$$S_1(a,b)\leq\binom{a+b+1}{a}-1.$$
\end{Lemma}

\begin{proof}
Without loss of generality, we assume that $\mathcal{P}=\{(A_i,B_i):1\leq i\leq m\}$ is a skew Bollob\'{a}s system with $|A_i|=a$ and $|B_i|=b$ for $i\in[m]$ such that $|\bigcup^{m}_{i=1}A_i|=S_1(a,b)$. Otherwise, if there is a skew Bollob\'{a}s system $\mathcal{P}'=\{(A_i',B_i'): 1\leq i \leq m\}$ with $|A_i'|\leq a$ and $|B_i'|\leq b$ for $i\in [m]$, we can always add new elements to every $A_i'$ and $B_i'$ to create a new skew Bollob\'{a}s system $\mathcal{P}''=\{(A_i'',B_i''): 1\leq i \leq m\}$ such that $|\bigcup^{m}_{i=1}A_i''|\geq|\bigcup^{m}_{i=1}A_i'|$.

Let $M_0=[m]$ and $\mathcal{P}_0=\mathcal{P}=\{(A^{(0)}_i,B^{(0)}_i): i\in M_0\}$ with $A^{(0)}_i=A_i$ and $B^{(0)}_i=B_i$ for $i\in M_0$. For $j=1,2,\ldots,a$, we define the index set $M_j\subseteq M_{j-1}$ and $A_i^{(j)}\subseteq A_i^{(j-1)}$ recursively (in increasing order of $j$) in the following way: if $1\leq j\leq a$, the index set $M_{j-1}$ and $A^{(j-1)}_i$ for every $i\in M_{j-1}$ are defined, then we let $M_{j}$ be a minimal subset of $M_{j-1}$ such that
$$\bigcup_{i\in M_{j}}A^{(j-1)}_i=\bigcup_{i\in M_{j-1}}A^{(j-1)}_i.$$
The minimality of $M_j$ implies that for every $i\in M_j$, there exists an element $x_i^{(j)}\in A_i^{(j-1)}$ such that $x_i^{(j)}\notin A_l^{(j-1)}$ for any $l\in M_j\setminus \{i\}$. Therefore we set
$$A_i^{(j)}=A_i^{(j-1)}\setminus\{x_i^{(j)}\}$$
for every $i\in M_j$. Note that $x_i^{(j)}\neq x_{i'}^{(j)}$ for any distinct $i,i'\in M_j$, and so we have
$$S_1(a,b)=\left|\bigcup^{m}_{i=1}A_i\right|=\sum_{j=1}^{a}|M_j|.$$

Now we shall choose appropriate $B^{(j-1)}_i$ such that $$\mathcal{P}_j=\{(A^{(j-1)}_i,B^{(j-1)}_i):i\in M_j\}$$
forms a skew Bollob\'{a}s system for every $j\in[a]$, by applying induction on $j$. The base case $\mathcal{P}_1=\{(A^{(0)}_i,B^{(0)}_i):i\in M_1\}$ is obviously a skew Bollob\'{a}s system. Suppose that for $1\leq j\leq a-1$, $\mathcal{P}_{j}=\{(A^{(j-1)}_i,B^{(j-1)}_i):i\in M_{j}\}$ is a skew Bollob\'{a}s system. Let $\mathcal{P}_{j+1}=\{(A^{(j)}_i,B^{(j)}_i):i\in M_{j+1}\}$ with $B^{(j)}_i=B^{(j-1)}_i$ for $i\in M_{j+1}$. If $\mathcal{P}_{j+1}$ forms a skew Bollob\'{a}s system, we are done. Otherwise, for any $u,v\in M_{j+1}(\subseteq M_{j})$, $u<v$, such that $A_u^{(j)}\cap B_v^{(j)}=\emptyset$, we describe a strategy below to adjust $B_v^{(j)}$.

The minimality of $M_j$ implies that there exists an element $x_u^{(j)}\in A_u^{(j-1)}$ for $u\in M_j$ such that $x_u^{(j)}\not\in A_l^{(j-1)}$ for any $l\in M_j\setminus\{u\}$. As $A_u^{(j)}=A_u^{(j-1)}\setminus\{x_u^{(j)}\}$, by the induction hypothesis, $x_u^{(j)}\in B_v^{(j-1)}$. Note that so far $B_v^{(j)}=B_v^{(j-1)}$, and so $x_u^{(j)}\in B_v^{(j)}$. The minimality of $M_{j+1}$ implies that there exists an element $x_u^{(j+1)}\in A_u^{(j)}$ such that $x_u^{(j+1)}\notin A_v^{(j)}$. Since $A_u^{(j)}\cap B_v^{(j)}=\emptyset$, $x_u^{(j+1)}\notin B_v^{(j)}$. We update $B_v^{(j)}$ by replacing its element $x_u^{(j)}$ with  $x_u^{(j+1)}$. As $M_{j+1}$ is finite for $j\in[a-1]$ the process will stop and we obtain a skew Bollob\'{a}s system $\mathcal{P}_{j+1}$.

Note that for every $0\leq j\leq a$, we have $|M_j|=|\mathcal{P}_j|$. Since $|A^{(j-1)}_i|=a-j+1$ and $|B^{(j-1)}_i|=b$ for $j\in[a]$ and $i\in M_j$, by Theorem \ref{skew-number}, we get $|\mathcal{P}_j|\leq \binom{a-j+1+b}{a-j+1}$, and hence
$$\sum_{j=1}^{a}|M_j|\leq\sum_{j=1}^{a}\binom{a-j+1+b}{a-j+1}=\sum_{k=1}^{a}\binom{k+b}{k}=\sum_{k=0}^{a}\binom{k+b}{k}-1=\binom{a+b+1}{a}-1.$$
This completes the proof.
\end{proof}

\begin{Lemma}\label{inequality2}
Let $a$ and $b$ be positive integers. Then
$$S_1(a,b)\geq S_1(a-1,b)+S_1(a,b-1)+1.$$
\end{Lemma}

\begin{proof}
Let $\mathcal{P}_1=\{(A_{i,1},B_{i,1}):1\leq i\leq g_1\}$ be a skew Bollob\'{a}s system on $X_1:=\bigcup_{i=1}^{g_1}(A_{i,1}\cup B_{i,1})$ with $|A_{i,1}|\leq a-1$ and $|B_{i,1}|\leq b$ such that $|\bigcup_{i=1}^{g_1}A_{i,1}|=S_1(a-1,b)$. Let $\mathcal{P}_2=\{(A_{i,2},B_{i,2}):1\leq i\leq g_2\}$ be a skew Bollob\'{a}s system on $X_2:=\bigcup_{i=1}^{g_2}(A_{i,2}\cup B_{i,2})$ with $|A_{i,2}|\leq a$ and $|B_{i,2}|\leq b-1$ such that $|\bigcup_{i=1}^{g_2}A_{i,2}|=S_1(a,b-1)$. Assume that $X_1\cap X_2=\emptyset$. Take a new point $x\notin X_1\cup X_2$. Let
\begin{align*}
\mathcal{P}=&\{(A_{i,1}\cup\{x\},B_{i,1}):1\leq i\leq g_1\}\cup\{(A_{i,2},B_{i,2}\cup\{x\}):1\leq i\leq g_2\}.
\end{align*}
Then $\mathcal{P}=\{(C_{i},D_{i}):1\leq i\leq g_1+g_2\}$ is a skew Bollob\'{a}s system with $|C_i|\leq a$ and $|D_i|\leq b$. We order the pairs in $\mathcal{P}$ so that all pairs of the form $(A_{i,1}\cup\{x\},B_{i,1})$ precede those of the form $(A_{i,2},B_{i,2}\cup\{x\})$. More precisely, we index the pairs in $\mathcal{P}$ as $\{(C_i,D_i): 1 \leq i \leq g_1+g_2\}$ where
$$
(C_i,D_i) =
\begin{cases}
(A_{i,1}\cup\{x\},B_{i,1}) & \text{for } 1 \leq i \leq g_1;\\
(A_{i-g_1,2},B_{i-g_1,2}\cup\{x\}) & \text{for } g_1+1 \leq i \leq g_1+g_2.
\end{cases}
$$
Indeed, for any $1 \leq i < j \leq g_1$, we have $C_i \cap D_j = (A_{i,1}\cup\{x\}) \cap B_{j,1} \neq \emptyset$ since $\mathcal{P}_1$ is a skew Bollob\'{a}s system. For any $g_1+1 \leq i < j \leq g_1+g_2$, we have $C_i \cap D_j = A_{i-g_1,2} \cap (B_{j-g_1,2}\cup\{x\}) \neq \emptyset$ since $\mathcal{P}_2$ is also a skew Bollob\'{a}s system. Finally, for any $1 \leq i \leq g_1$ and $g_1+1 \leq j \leq g_1+g_2$, we have $C_i \cap D_j = (A_{i,1}\cup\{x\}) \cap (B_{j-g_1,2}\cup\{x\}) \neq \emptyset$, as both sets contain the element $x$. Note that $|\bigcup_{i=1}^{g_1+g_2}C_{i}|=S_1(a-1,b)+S_1(a,b-1)+1.$
\end{proof}
	
\begin{proof}[{\bf Proof of Theorem $\ref{main-rssult}$.}]
By the definition of $S_1(a,b)$, it is easy to see that $S_1(0,b)=0$ and $S_1(a,0)=a$. The upper bound follows from Lemma \ref{upper bound}. To establish the lower bound $$S_1(a,b) \geq \binom{a+b+1}{a}-1,$$
we apply induction on $a$ and $b$. The conclusion holds for the case $a=0$ and $b\geq 0$, and the case $b=0$ and $a\geq 0$.  Let us suppose that the statement holds for all $(a-1,b)$ and $(a,b-1)$ with $a\geq1$ and $b\geq1$. Apply the induction hypothesis and Lemma \ref{inequality2}. We have
\begin{equation}
\begin{aligned}
S_1(a,b)\geq & S_1(a-1,b)+S_1(a,b-1)+1\\ \nonumber
\geq &\binom{a-1+b+1}{a-1}-1+\binom{a+b-1+1}{a}-1+1\\ \nonumber
=&\binom{a+b}{a-1}+\binom{a+b}{a}-1\\ \nonumber
=&\binom{a+b+1}{a}-1.
\end{aligned}
\end{equation}
This completes the proof.
\end{proof}

\begin{Lemma}\label{S-and-T}
Let $a$ and $b$ be non-negative integers. Then $S_2(a,b)=S_1(b,a)$.
\end{Lemma}

\begin{proof}
Let $\{(A_i, B_i):1\leq i\leq m\}$ be a skew Bollob\'{a}s system with $|A_i|\leq a$ and $|B_i|\leq b$ for $i\in[m]$, such that $|\bigcup_{i=1}^{m} B_i|=S_2(a,b)$. Let $A'_i=B_{m+1-i}$ and $B'_i=A_{m+1-i}$ for $i\in [m]$. Then $\{(A'_i, B'_i): 1\leq i\leq m\}$ forms a skew Bollob\'{a}s system with $|A'_i|\leq b$ and $|B'_i|\leq a$ for $i\in [m]$, and so $|\bigcup_{i=1}^{m} A'_i|\leq S_1(b,a)$. Note that $|\bigcup_{i=1}^{m} B_i|=|\bigcup_{i=1}^{m} A'_i|$. It follows that $S_2(a,b)\leq S_1(b,a)$. Similarly, $S_1(b,a)\leq S_2(a,b)$.
\end{proof}

\begin{proof}[{\bf Proof of Theorem $\ref{2nd-result}$.}]
It follows from Theorem \ref{main-rssult} and Lemma \ref{S-and-T} that
$$S_2(a,b)=S_1(b,a)=\binom{b+a+1}{b}-1=\binom{a+b+1}{a+1}-1.$$
\end{proof}
	
\section{Concluding remarks}

We determine the exact values of $S_1(a,b)$ and $S_2(a,b)$ in this paper. A related concept was studied in \cite{fwf}. Specifically, let $n_{\text{skew}}(a, b)$ be the maximum size of the ground set of skew Bollob\'{a}s systems. That is, $n_{\text{skew}}(a,b)=\max\{|\bigcup_{i=1}^{m} (A_i\cup B_i)|: \{(A_i,B_i): 1\leq i\leq m\}$ is a skew Bollob\'{a}s system with $|A_i|\leq a$, $|B_i|\leq b$ for $i\in [m]\}$.
It was shown in \cite{fwf} that
$$
n_{\text{skew}}(a, b) = \binom{a+b+2}{a+1} - \binom{a+b}{a} - 1
$$
holds for all non-negative integers $a$ and $b$. We remark that $S_1(a,b)+S_2(a,b)\neq n_{\text{skew}}(a, b)$ for positive integers $a$ and $b$. Indeed, $S_1(a,b)+S_2(a,b)=\binom{a+b+1}{a}-1+\binom{a+b+1}{a+1}-1=\binom{a+b+2}{a+1}-2=n_{\text{skew}}(a, b)+\binom{a+b}{a}-1$.

A \emph{Bollob\'{a}s system} $\mathcal{P}=\{(A_i,B_i): 1\leq i\leq m\}$ is a collection of pairs of disjoint subsets of $[n]$ such that $A_i\cap B_j\ne\emptyset$ for any $1\leq i, j\leq m$ and $i\ne j$ (cf. \cite{B}). Similarly to $n_{\text {skew}}(a,b)$ and $S_1(a,b)$, Tuza \cite{tuza1985} introduced two related functions $n(a,b)=\max\{|\bigcup_{i=1}^{m} (A_i\cup B_i)|: \{(A_i,B_i): 1\leq i\leq m\}$ is a Bollob\'{a}s system with $|A_i|\leq a$, $|B_i|\leq b$ for $i\in [m]\}$ and $n'(a,b)=\max\{|\bigcup_{i=1}^{m} A_i|: \{(A_i,B_i): 1\leq i\leq m\}$ is a Bollob\'{a}s system with $|A_i|\leq a$, $|B_i|\leq b$ for $i\in [m]\}$. The lower and upper bounds for $n(a,b)$ and $n'(a,b)$ were investigated in \cite{m,np,tuza1985}. However, the exact values of $n(a,b)$ and $n'(a,b)$ are still unknown. The approach in this note does not seem applicable to deriving improved bounds for $n(a,b)$ and $n'(a,b)$.


\begin{thebibliography}{99}


\bibitem{B}
B.~Bollob\'{a}s, On generalized graphs, {\it Acta Math. Acad. Sci. Hungar.}, {\bf 16} (1965), 447--452.

\bibitem{c}
A.~Calbet, $K_r$-saturated graphs and the two families theorem, {\it Electron. J. Combin.}, {\bf 31} (2024), $\#$P4.68.

\bibitem{fwf}
Y.~Fang, X.~Wang and T.~Feng, A note on the maximum size of the ground set of skew Bollob\'{a}s systems, {\it Discrete Math.}, {\bf 348} (2025), 114650.
		
\bibitem{frankl}
P.~Frankl, An extremal problem for two families of sets, {\it European J. Comb.}, {\bf 3} (1982), 125--127.
		
\bibitem{gp}
D.~Gerbner and B.~Patk\'{o}s, {\it Extremal Finite Set Theory}, Chapman and Hall/CRC (2018).

\bibitem{kalai}
G.~Kalai, Intersection patterns of convex sets, {\it Israel J. Math.}, {\bf 48} (1984), 161--174.
		
\bibitem{m}
K.~Majumder, On the maximum number of points in a maximal intersecting family of finite sets, {\it Combinatorica}, {\bf 37} (2017), 87--97.

\bibitem{np}
Z.L.~Nagy and B.~Patk\'{o}s, On the number of maximal intersecting $k$-uniform families and further applications of Tuza's set pair method, {\it Electron. J. Comb.}, {\bf 22} (2015), $\#$P1.83.

\bibitem{tuza1985}
Z.~Tuza, Critical hypergraphs and intersecting set-pair systems, {\it J. Comb. Theory Ser. B}, {\bf 39} (1985), 134--145.		
\end{thebibliography}
\end{document}